\documentclass[11pt]{article}
\usepackage[utf8]{inputenc}
\usepackage[a4paper,margin=1in]{geometry}
\usepackage{amsmath,amssymb,amsthm}
\usepackage{xcolor}
\usepackage{hyperref}
\usepackage{microtype}
\usepackage{comment}
\hypersetup{
 colorlinks=true,
 linkcolor=blue!50!black,
 citecolor=blue!50!black,
 urlcolor=blue!50!black
}

\newtheorem{theorem}{Theorem}[section]
\newtheorem{lemma}[theorem]{Lemma}
\newtheorem{corollary}[theorem]{Corollary}
\newtheorem{proposition}[theorem]{Proposition}
\newtheorem{remark}[theorem]{Remark}

\newcounter{introresult}

\newtheorem{introtheorem}[introresult]{Theorem}
\newtheorem{introcorollary}[introresult]{Corollary}
\newtheorem*{introcorollary*}{Corollary}

\newcommand{\rad}{\operatorname{rad}}

\title{Triple torsion, triple cup products, and embedding obstructions for rational homology 3-spheres}
\author{Weizhe Niu}
\date{}

\begin{document}
\maketitle

\begin{abstract} Freedman and Krushkal introduced a triple torsion linking form for rational homology \(3\)-spheres and used it to obstruct locally flat embeddings in \(S^4\). For every odd prime \(p\), we identify their triple torsion form, computed with parameter \(t=p\) on rational homology \(3\)-spheres whose first homology has exponent \(p\), with the mod-\(p\) triple cup product under torsion-linking duality. For algebraically split \(\pm p\)-framed surgery links, this gives a signed formula in terms of Milnor's integral length-three invariants \(\bar\mu_{ijk}\), with the framing-sign factor dictated by torsion-linking duality. We then use Borromean band-sums to realize arbitrary mod-\(p\) triple cup tensors on rational homology \(3\)-spheres with \( H_1\cong(\mathbb Z/p)^6 \) and fixed hyperbolic ordinary torsion linking form. Finally, using the classical spinor/Klein model for the split six-dimensional quadratic space, we classify the tensors with no dual null Hantzsche pair. This produces, for every odd prime \(p\), a rational homology \(3\)-sphere with hyperbolic ordinary torsion linking form but with no locally flat embedding in \(S^4\), and indeed no locally flat embedding in any integer homology \(4\)-sphere. 
\end{abstract}


\section{Introduction}

Which closed orientable \(3\)-manifolds embed in \(S^4\) is a classical
problem in low-dimensional topology.  Freedman's work implies that every
integral homology \(3\)-sphere embeds topologically in \(S^4\)
\cite{Freedman82}.  For rational homology \(3\)-spheres which are not
integral homology spheres, torsion in first homology gives genuine
embedding obstructions.

The first obstruction is Hantzsche's classical linking-form obstruction
\cite{Hantzsche}.  If a rational homology \(3\)-sphere \(M\) embeds in
\(S^4\), then its \textit{ordinary torsion linking form}
\[
\lambda_2\colon H_1(M;\mathbb Z)\times H_1(M;\mathbb Z)
\longrightarrow \mathbb Q/\mathbb Z
\]
is hyperbolic.  Equivalently, \(H_1(M;\mathbb Z)\) decomposes as a direct
sum of two dual isotropic summands.  This condition is necessary for
embedding, but it is not sufficient.

Freedman and Krushkal introduced a \textit{triple torsion linking form}
\(\lambda_3\) for rational homology \(3\)-spheres \cite{FK}.  The invariant
is defined on an asymmetric domain: \(\lambda_3(x,y,z)\) is defined when
\[
\lambda_2(x,y)=\lambda_2(x,z)=0.
\]
In particular, it is defined on triples contained in a Hantzsche
Lagrangian.  They proved that if a rational homology \(3\)-sphere embeds
locally flatly in \(S^4\), then \(\lambda_3\) vanishes on each of the two
Hantzsche Lagrangians arising from the complementary regions; the same
obstruction applies to locally flat embeddings in integer homology
\(4\)-spheres \cite[Theorem~1.2 and the following remark]{FK}.  They constructed an example with first homology \((\mathbb Z/3)^6\),
hyperbolic ordinary torsion linking form, and nonzero triple torsion
obstruction.

The first point of this paper is that, in exponent \(p\), the Freedman--Krushkal invariant is a familiar cohomological object. The mod-\(p\) triple cup product is classical in the study of cohomology rings of \(3\)-manifolds, going back to Sullivan and Turaev \cite{Sullivan,TuraevCohomologyRings}. Motivated by the questions in Freedman--Krushkal's final section, we give an exponent-\(p\) cup-product interpretation of their invariant and, for algebraically split \(\pm p\)-framed surgery links, a signed Milnor \(\bar\mu_{ijk}\)-formula.

Let \(p\) be an odd prime and suppose \(H_1(M;\mathbb Z)\) has exponent
\(p\).  We compute \(\lambda_3\) with Freedman--Krushkal exponent
\(t=p\), and write
\[
\kappa_p\colon \frac1p\mathbb Z/\mathbb Z\longrightarrow\mathbb F_p,
\qquad
\kappa_p\left(\frac ap\right)=a\bmod p.
\]
For \(x\in H_1(M;\mathbb Z)\), define the linking-dual cohomology class
\(
x^\lambda\in H^1(M;\mathbb F_p)
\)
by
\[
x^\lambda(u)=\kappa_p(\lambda_2(x,u)).
\]

\begin{introtheorem}[FK/cup-product identification]
\label{thm:intro-A}
Let \(M\) be a closed oriented rational homology \(3\)-sphere whose first
homology \(H_1(M;\mathbb Z)\) has exponent \(p\), where \(p\) is an odd prime. For $x,y,z\in H_1(M;\mathbb Z)$ with 
\(
\lambda_2(x,y)=\lambda_2(x,z)=0,
\)
we have
\[
\kappa_p\bigl(\lambda_3^p(x,y,z)\bigr)
=
\left\langle x^\lambda\cup y^\lambda\cup z^\lambda,[M]\right\rangle.
\]

\end{introtheorem}

Our second theorem is a surgery interpretation.  For algebraically split
surgery links, the same tensor is Milnor's length-three tensor, with the
framing signs forced by torsion-linking duality.

\begin{introtheorem}[Surgery/Milnor formula]
\label{thm:intro-B}
Let
\[
L=L_1\cup\cdots\cup L_m\subset S^3
\]
be an oriented algebraically split framed link with framings
\(
\varepsilon_i p, \varepsilon_i=\pm1,
\)
and let \(M=S^3_L\).  Let \(m_i\in H_1(M;\mathbb Z)\) be the meridian
classes, and let \(\bar\mu_{ijk}(L)\) be Milnor's integral length-three
invariant.  Then, for \(i<j<k\),
\[
\kappa_p\bigl(\lambda_3^p(m_i,m_j,m_k)\bigr)
=
 s_\mu\,\varepsilon_i\varepsilon_j\varepsilon_k\,
 \bar\mu_{ijk}(L)
\quad\text{in }\mathbb F_p,
\]
where \(s_\mu\in\{\pm1\}\) is the global sign determined by the chosen
Milnor and triple-cup conventions.  More generally, if
\(x_r=\sum_i a_{ri}m_i\), \(r=1,2,3\), satisfy the FK domain condition,
then
\[
\kappa_p\bigl(\lambda_3^p(x_1,x_2,x_3)\bigr)
=
 s_\mu
 \sum_{i<j<k}
 \varepsilon_i\varepsilon_j\varepsilon_k\,
 \bar\mu_{ijk}(L)
 \det(a_{ri},a_{rj},a_{rk}).
\]
\end{introtheorem}

We now specialize to the minimal hyperbolic rank in which a triple
obstruction can occur, namely \(H_1\cong(\mathbb Z/p)^6\). In this case, we realize arbitrary triple cup tensors while keeping the ordinary
linking form fixed. Put
\[
\varepsilon_i=\begin{cases}
+1,&1\le i\le3,\\
-1,&4\le i\le6.
\end{cases}
\]

\begin{introtheorem}[Realization of triple cup tensors]
\label{thm:intro-C}
Let \(p\) be an odd prime.  For every alternating tensor
\[
\tau=\sum_{1\le i<j<k\le6}\tau_{ijk}\,
\alpha_i\wedge\alpha_j\wedge\alpha_k
\]
on a six-dimensional \(\mathbb F_p\)-vector space with basis
\(\alpha_1,\ldots,\alpha_6\), there exists a closed oriented rational
homology \(3\)-sphere \(M_{p,\tau}\) such that
\[
H_1(M_{p,\tau};\mathbb Z)\cong(\mathbb Z/p)^6
\]
and
\[
\lambda_2(M_{p,\tau})\cong
\frac1p\langle1,1,1,-1,-1,-1\rangle.
\]
If \(m_i\) is the meridian basis for \(H_1(M_{p,\tau};\mathbb Z)\) and
\(\alpha_i\in H^1(M_{p,\tau};\mathbb F_p)\) is the dual cohomology basis,
then
\[
\left\langle
\alpha_i\cup\alpha_j\cup\alpha_k,[M_{p,\tau}]
\right\rangle
=
\tau_{ijk}
\quad (i<j<k).
\]

\end{introtheorem}

The Freedman--Krushkal obstruction then has a clean cup-product formulation. For
\(A\subset H_1(M;\mathbb Z)\), we let
\[
A^\lambda=\{x^\lambda:x\in A\}\subset H^1(M;\mathbb F_p).
\]

\begin{introtheorem}[Cup-product obstruction criterion]
\label{thm:intro-D}
Let \(M\) be a rational homology \(3\)-sphere whose first homology \(H_1(M;\mathbb Z)\) has
exponent \(p\).  If \(M\) embeds locally flatly in \(S^4\), or in an
integer homology \(4\)-sphere, then the two Hantzsche Lagrangians
\(A,B\subset H_1(M;\mathbb Z)\) coming from the complementary regions
satisfy
\[
\left\langle u\cup v\cup w,[M]\right\rangle=0
\]
for all \(u,v,w\in A^\lambda\), and also for all \(u,v,w\in B^\lambda\).
Consequently, if the triple cup tensor of \(M\) has no dual null
Hantzsche pair under linking duality, then \(M\) admits no such embedding.
\end{introtheorem}

Finally, we classify the no-dual-null condition under the classical
spinor/Klein model for the split six-dimensional quadratic space.  Let
\(W=\mathbb F_p^4\) and \(V=\bigwedge^2W\), equipped with the wedge-pairing
quadratic form.  The maximal isotropic \(3\)-planes form the two Klein
rulings
\[
A_{[x]}=x\wedge W,
\qquad
B_{[\varphi]}=\bigwedge^2(\ker\varphi).
\]
We consider the standard decomposition
\[
\bigwedge^3(\bigwedge^2W)^*
\cong
\operatorname{Sym}^2(W^*)\oplus\operatorname{Sym}^2(W)
\]
from representation theory which we will also verify by an
elementary coordinate calculation.  The point is to classify the specific
no-dual-null condition relevant to the FK obstruction.

\begin{introtheorem}[Spinor/Klein no-dual-null criterion]
\label{thm:intro-E}
Let \(p\) be an odd prime, let \(W=\mathbb F_p^4\), and set
\(V=\bigwedge^2W\) with the wedge-pairing quadratic form.  Let
\(
\omega\in\bigwedge^3V^*
\)
correspond to
\[
(Q,Q^\vee)
\in
\operatorname{Sym}^2(W^*)\oplus\operatorname{Sym}^2(W),
\]
where \(Q\) is a quadratic form on \(W\), and \(Q^\vee\) is a quadratic
form on \(W^*\).  Then \(\omega\) has no dual null Lagrangian pair if and
only if there is a \(2\)-plane \(R\subset W\) such that
\[
\rad(Q)=R,
\]
the induced binary quadratic form on \(W/R\) is anisotropic,
\[
\rad(Q^\vee)=R^\perp\subset W^*,
\]
and the induced binary quadratic form on \(W^*/R^\perp\) is anisotropic.
Equivalently,
\[
Z(Q)=\mathbb P(R)\subset \mathbb P(W),
\qquad
Z(Q^\vee)=\mathbb P(R^\perp)\subset \mathbb P(W^*).
\]
\end{introtheorem}

\begin{introcorollary}[All odd-prime nonembedding examples]
\label{cor:intro-F}
For every odd prime \(p\), there exists a closed oriented rational
homology \(3\)-sphere \(M_p\) such that
\(
H_1(M_p;\mathbb Z)\cong(\mathbb Z/p)^6
\)
and
\[
\lambda_2(M_p)
\cong
\frac1p\langle1,1,1,-1,-1,-1\rangle
\cong H\oplus H\oplus H,
\]
where \(H\) denotes the standard hyperbolic linking form on
\((\mathbb Z/p)^2\). In addition, \(M_p\) admits no locally flat
embedding in \(S^4\), and no locally flat embedding in any integer
homology \(4\)-sphere.  

\end{introcorollary}

These examples occur in the smallest possible elementary abelian rank in
which a triple obstruction can be nontrivial: a Hantzsche Lagrangian must
have rank at least \(3\), so the ambient elementary abelian group must
have rank at least \(6\).

\medskip
\noindent\textbf{Organization.}
Section~\ref{sec:FK} recalls the Freedman--Krushkal obstruction.
Section~\ref{sec:cup} proves the FK/cup-product identification and the
surgery/Milnor formula.  Section~\ref{sec:realization} realizes arbitrary
triple cup tensors with fixed hyperbolic torsion linking form.  Section~\ref{sec:criterion} reformulates the embedding obstruction in cup-product
language.  Section~\ref{sec:klein} proves the spinor/Klein criterion and
deduces the all odd-prime examples.

\section{The Freedman--Krushkal obstruction}
\label{sec:FK}

Let \(M\) be a rational homology \(3\)-sphere.  Its ordinary torsion
linking form is the nonsingular pairing
\[
\lambda_2\colon H_1(M;\mathbb Z)\times H_1(M;\mathbb Z)
\longrightarrow \mathbb Q/\mathbb Z.
\]
To simplify our notations, we will often use the same symbol for a homology class and for a chosen oriented cycle representing it. We use the convention that if \(a,b\in H_1(M;\mathbb Z)\), \(p a=0\), and
\(A\) is a singular \(2\)-chain with \(\partial A=pa\), then
\[
\lambda_2(a,b)=\frac1p(A\cdot b)\in\mathbb Q/\mathbb Z.
\]

A subgroup \(L\le H_1(M;\mathbb Z)\) is called a Hantzsche Lagrangian if
\[
\lambda_2|_{L\times L}=0,
\qquad
|L|^2=|H_1(M;\mathbb Z)|.
\]
A dual Hantzsche pair is a direct sum decomposition
\[
H_1(M;\mathbb Z)=A\oplus B
\]
where \(A\) and \(B\) are Hantzsche Lagrangians.  Hantzsche's classical
obstruction says that if a rational homology \(3\)-sphere embeds in
\(S^4\), then its torsion linking form admits such a dual pair
\cite{Hantzsche}. Freedman and Krushkal define a triple torsion linking form \(\lambda_3\)
under the domain condition
\[
\lambda_2(x,y)=\lambda_2(x,z)=0.
\]
In particular, \(\lambda_3(x,y,z)\) is defined for triples in a Hantzsche
Lagrangian. It depends on the choice of a fixed exponent \(t\) annihilating
\(H_1(M;\mathbb Z)\), and takes values in \( \frac1t\mathbb Z/\mathbb Z\subset \mathbb Q/\mathbb Z \) \cite[Definition~3.1, Remark~3.2(2)]{FK}.  When
\(H_1(M;\mathbb Z)\) has exponent \(p\), we set \(t=p\)
and write \(\lambda_3^p\). 

We recall the following theorem of Freedman--Krushkal \cite[Theorem~1.2 and the following remark]{FK}.

\begin{theorem}[Freedman--Krushkal]
\label{thm:FK}
Let \(M\) be a rational homology \(3\)-sphere.  If \(M\) admits a locally flat embedding in \(S^4\) with two complementary regions \(X\) and \(Y\), then the two kernels
\[
A=\ker\{H_1(M)\to H_1(X)\},
\qquad
B=\ker\{H_1(M)\to H_1(Y)\},
\]
form a dual Hantzsche pair and satisfy
\[
\lambda_3|_A=0,
\qquad
\lambda_3|_B=0.
\]
The same conclusion holds for a locally flat embedding in any integer
homology \(4\)-sphere.
\end{theorem}

\section{The triple cup product interpretation}
\label{sec:cup}

Throughout this section \(p\) is an odd prime.  Let \(M\) be a closed
oriented rational homology \(3\)-sphere whose first homology has exponent
\(p\). We fix an isomorphism
\[
\kappa_p:\frac1p\mathbb Z/\mathbb Z\longrightarrow \mathbb F_p,
\qquad
\kappa_p\left(\frac ap\right)=a\bmod p.
\]
For \(x\in H_1(M;\mathbb Z)\), define
\(
x^\lambda\in H^1(M;\mathbb F_p)
\)
by
\[
x^\lambda(u)=\kappa_p(\lambda_2(x,u)),
\qquad u\in H_1(M;\mathbb Z).
\]
The universal coefficient
theorem gives
\[
H^1(M;\mathbb F_p)\cong \operatorname{Hom}(H_1(M;\mathbb Z),\mathbb F_p).
\]

\subsection{The FK/cup-product identification}

\begin{lemma}
\label{lem:first-stage-PD}
Let \(x\in H_1(M;\mathbb Z)\), and let \(\Sigma\) be a singular
\(2\)-chain with
\(
\partial\Sigma=p x.
\)
Let
\(\bar\Sigma\in C_2(M;\mathbb F_p)\) be its image under the coefficient
reduction \(C_2(M;\mathbb Z)\to C_2(M;\mathbb F_p)\). Then
\(
\bar\Sigma\in Z_2(M;\mathbb F_p),
\)
and
\[
\operatorname{PD}_{\mathbb F_p}[\bar\Sigma]=x^\lambda
\in H^1(M;\mathbb F_p).
\]
Equivalently, for every \(u\in H_1(M;\mathbb Z)\),
\[
[\bar\Sigma]\cdot u=\kappa_p(\lambda_2(x,u)).
\]
\end{lemma}

\begin{proof}
Since \(\partial\Sigma=px\), reducing modulo \(p\) gives
\(\partial\bar\Sigma=0\).  Thus \(\bar\Sigma\) is a mod-\(p\) cycle. Let \(u\in H_1(M;\mathbb Z)\), represented by an oriented curve disjoint
from \(\partial\Sigma\).  By the geometric definition of the ordinary
torsion linking form,
\[
\lambda_2(x,u)=\frac1p(\Sigma\cdot u)\in\mathbb Q/\mathbb Z.
\]
Therefore
\[
[\bar\Sigma]\cdot u
=
\Sigma\cdot u\bmod p
=
\kappa_p(\lambda_2(x,u))
=
x^\lambda(u).
\]
This is precisely the assertion that
\[
\operatorname{PD}_{\mathbb F_p}[\bar\Sigma]=x^\lambda.
\]
\end{proof}

\begin{lemma}
\label{lem:one-boundary}
Let \(x\in H_1(M;\mathbb Z)\) be represented by an oriented curve \(c\).
If \(p x=0\), then there is a compact connected oriented singular surface
\(
i:\Sigma\longrightarrow M
\)
with exactly one boundary component, such that \(i|_{\partial\Sigma}\) maps
to \(c\) with degree \(p\). Equivalently, as a singular chain,
\(
\partial\Sigma=p c.
\)
\end{lemma}

\begin{proof} Since \(p[c]=0\), choose an integral singular \(2\)-chain \(C\) with \(\partial C=p c\). By the standard geometric realization argument for relative singular \(2\)-chains, \(C\) may be represented by a compact oriented singular surface, possibly disconnected and with several boundary components, whose total boundary maps to \(c\) with degree \(p\). Connecting components by tubes mapped to a point, and then attaching boundary \(1\)-handles mapped into a collar of \(c\), does not change the total boundary chain. The resulting surface is connected and has one boundary component mapping to \(c\) with degree \(p\). \end{proof}

\begin{theorem}[The Freedman--Krushkal form as a triple cup product]
\label{thm:FK-cup}
Let \(M\) be a closed oriented rational homology \(3\)-sphere whose first
homology has exponent \(p\), and compute \(\lambda_3\) with exponent
\(t=p\).  Let
\(
x,y,z\in H_1(M;\mathbb Z)
\)
satisfy
\(
\lambda_2(x,y)=\lambda_2(x,z)=0.
\)
Then
\[
\kappa_p\bigl(\lambda_3^p(x,y,z)\bigr)
=
\left\langle
x^\lambda\cup y^\lambda\cup z^\lambda,[M]
\right\rangle
\in \mathbb F_p.
\]

\end{theorem}

\begin{proof}
Choose oriented curve representatives, again denoted \(x,y,z\). By Lemma~\ref{lem:one-boundary}, choose a compact connected oriented singular surface \( i:\Sigma\longrightarrow M\) with \(\partial\Sigma=p x, \) with one boundary component. Under the domain condition \(\lambda_2(x,y)=\lambda_2(x,z)=0, \) the Freedman--Krushkal construction allows us to modify \(\Sigma\) in its
interior, relative to \(\partial\Sigma\), so that its image is disjoint
from the chosen representatives of \(y\) and \(z\).  We replace \(\Sigma\) by this modified surface.  The one-boundary condition and the boundary
degree \(p\) are preserved.

Let \(\bar\Sigma\) be the image of \(\Sigma\) under coefficient reduction modulo \(p\). Since \(\partial\Sigma=p x\), we have \(\partial\bar\Sigma=0\), so \(\bar\Sigma\) is a mod-\(p\) \(2\)-cycle. By Lemma~\ref{lem:first-stage-PD}, \( \operatorname{PD}_{\mathbb F_p}[\bar\Sigma]=x^\lambda. \)

We next compare the Freedman--Krushkal surface formula with the cup
product pairing on \(\Sigma\). Since \(\Sigma\) is connected and has one
boundary component, the natural map
\[
j:H^1(\Sigma,\partial\Sigma;\mathbb F_p)
\longrightarrow
H^1(\Sigma;\mathbb F_p)
\]
is an isomorphism. For \(\alpha,\beta\in H^1(\Sigma;\mathbb F_p)\), we
write
\[
\left\langle \alpha\cup\beta,[\Sigma,\partial\Sigma]\right\rangle
\]
for
\[
\left\langle
j^{-1}(\alpha)\cup\beta,[\Sigma,\partial\Sigma]
\right\rangle,
\]
where the cup product is the relative cup product
\[
H^1(\Sigma,\partial\Sigma;\mathbb F_p)
\otimes
H^1(\Sigma;\mathbb F_p)
\longrightarrow
H^2(\Sigma,\partial\Sigma;\mathbb F_p).
\]
Under this convention, if
\(
\{\gamma_\ell,\delta_\ell\}_{\ell=1}^g
\)
is a symplectic basis for \(H_1(\Sigma;\mathbb Z)\), ordered and oriented
so that
\(
\gamma_\ell\cdot\delta_\ell=+1,
\)
then the pairing is the standard Poincaré--Lefschetz cup pairing on a once-bordered surface
\[
\left\langle
\alpha\cup\beta,[\Sigma,\partial\Sigma]
\right\rangle
=
\sum_{\ell=1}^g
\left(
\alpha(\gamma_\ell)\beta(\delta_\ell)
-
\alpha(\delta_\ell)\beta(\gamma_\ell)
\right).
\]

Let \(\eta\in H_1(\Sigma;\mathbb Z)\), and choose a singular \(2\)-chain
\(C_\eta\) in \(M\) with
\(
\partial C_\eta=p\,i_*\eta.
\)
Then
\[
C_\eta\cdot y\bmod p
=
\kappa_p\bigl(\lambda_2(i_*\eta,y)\bigr)
=
\kappa_p\bigl(\lambda_2(y,i_*\eta)\bigr)
=
y^\lambda(i_*\eta)
=
(i^*y^\lambda)(\eta),
\]
where we used the symmetry of the ordinary linking form.  Similarly,
\[
C_\eta\cdot z\bmod p=(i^*z^\lambda)(\eta).
\]

Choose second-stage chains \(C_\ell,D_\ell\) in \(M\) with \[ \partial C_\ell=p\gamma_\ell, \qquad \partial D_\ell=p\delta_\ell, \] where, as usual, \(\gamma_\ell\) and \(\delta_\ell\) also denote their image cycles in \(M\).
Freedman--Krushkal's formula \cite[Equation~(3.5)]{FK}, with \(t=p\), gives
\[
\lambda_3^p(x,y,z)
=
\frac1p
\sum_{\ell=1}^g
\left(
(C_\ell\cdot y)(D_\ell\cdot z)
-
(C_\ell\cdot z)(D_\ell\cdot y)
\right)
\in\mathbb Q/\mathbb Z.
\]
Applying \(\kappa_p\) and using the preceding congruences for
\(\eta=\gamma_\ell,\delta_\ell\), we obtain
\[
\kappa_p\bigl(\lambda_3^p(x,y,z)\bigr)
=
\sum_{\ell=1}^g
\left(
(i^*y^\lambda)(\gamma_\ell)(i^*z^\lambda)(\delta_\ell)
-
(i^*y^\lambda)(\delta_\ell)(i^*z^\lambda)(\gamma_\ell)
\right).
\]
By the symplectic-basis formula, this is
\[
\left\langle
 i^*y^\lambda\cup i^*z^\lambda,[\Sigma,\partial\Sigma]
\right\rangle.
\]
The mod-\(p\) reduction of the pushforward of a relative fundamental
chain for \((\Sigma,\partial\Sigma)\) is precisely the absolute cycle
\(\bar\Sigma\), and hence represents
\(
[\bar\Sigma]\in H_2(M;\mathbb F_p).
\)
By naturality of the cap product for this relative-to-absolute pushforward,
\[
\left\langle
 i^*(y^\lambda\cup z^\lambda),[\Sigma,\partial\Sigma]
\right\rangle
=
\left\langle
 y^\lambda\cup z^\lambda,[\bar\Sigma]
\right\rangle.
\]
By the defining property of Poincaré duality,
\[
\left\langle
 y^\lambda\cup z^\lambda,[\bar\Sigma]
\right\rangle
=
\left\langle
 y^\lambda\cup z^\lambda\cup \operatorname{PD}_{\mathbb F_p}[\bar\Sigma],
 [M]
\right\rangle.
\]
Since \(\operatorname{PD}_{\mathbb F_p}[\bar\Sigma]=x^\lambda\), this is
\[
\left\langle
 y^\lambda\cup z^\lambda\cup x^\lambda,[M]
\right\rangle.
\]
Since \(y^\lambda\cup z^\lambda\) has degree \(2\), moving
\(x^\lambda\) to the front introduces the sign \((-1)^{2\cdot1}=+1\).
Thus
\[
\left\langle
 y^\lambda\cup z^\lambda\cup x^\lambda,[M]
\right\rangle
=
\left\langle
 x^\lambda\cup y^\lambda\cup z^\lambda,[M]
\right\rangle.
\]
Combining the displayed equalities proves the theorem.

\end{proof}

\begin{remark}
The proof used only
\[
\lambda_2(x,y)=\lambda_2(x,z)=0.
\]
It did not use \(\lambda_2(y,z)=0\).  Thus the theorem is valid on the
asymmetric Freedman--Krushkal domain.  If \(x,y,z\) lie in a common
isotropic subgroup, then all three pairwise ordinary linking numbers
vanish, all permutations are defined, and the formula above shows that
the restriction is alternating because the mod-\(p\) triple cup product of
degree-one classes is alternating for odd \(p\).  This agrees with the
antisymmetry proved by Freedman--Krushkal \cite[Proposition~6.2]{FK}.
\end{remark}

\subsection{Algebraically split surgery links and Milnor invariants}
\label{subsec:milnor}
Let
\[
L=L_1\cup\cdots\cup L_m\subset S^3
\]
be an oriented algebraically split framed link, i.e.~the pairwise linking numbers vanish, with framings
\(
\varepsilon_i p\),
 \(\varepsilon_i=\pm1.
\)
Let \(M=S^3_L\), and let \(m_i\in H_1(M;\mathbb Z)\) be the meridian
class of \(L_i\).  Let \(\alpha_i\in H^1(M;\mathbb F_p)\) be the
cohomology basis dual to the meridian basis:
\[
\alpha_i(m_j)=\delta_{ij}.
\]
Milnor's \(\bar\mu\)-invariants were introduced in
\cite{Milnor54,Milnor57}; their relation with Massey products in link
exteriors is treated by Turaev \cite{TuraevMassey}. Because \(L\) is
algebraically split, the length-three invariant \(\bar\mu_{ijk}(L)\) is
integrally well-defined for distinct \(i,j,k\): the usual indeterminacy is
generated by the pairwise linking numbers among the three components, and
those linking numbers vanish.

The following proposition is the length-three algebraically split case of
the classical relation between surgery, triple cup products, and Milnor
invariants.  Turaev's finite-coefficient cohomology-ring formula gives this
relation in general, and Cochran--Melvin recall the same input, noting that triple cup-product forms can be calculated from the
triple Milnor invariants of three-component sublinks of a surgery
presentation \cite[Theorem~1.6]{CochranMelvin}.
The zero-surgery version also appears in the form used by
Cochran--Gerges--Orr \cite[Corollary~3.5]{CGO}.
For completeness, we recall the capping argument identifying the relevant
cohomology classes and triple intersections.  The final identification with
Milnor's invariant is the length-three case of Turaev's surgery formula;
Mellor--Melvin \cite[Theorem~1]{MellorMelvin} give useful geometric
background for the corresponding Seifert-surface interpretation.

\begin{proposition}[Turaev--Milnor surgery formula in length three]
\label{prop:Turaev-Milnor}
There is a global sign
\(
s_\mu\in\{\pm1\},
\)
depending only on the convention for Milnor's \(\bar\mu\)-invariants and on
the orientation convention for triple cup products, such that for all
\(i<j<k\),
\[
\left\langle
\alpha_i\cup\alpha_j\cup\alpha_k,[M]
\right\rangle
=
 s_\mu\,\bar\mu_{ijk}(L)
\quad\text{in }\mathbb F_p.
\]
\end{proposition}

\begin{proof}
Let \(N(L)=\bigsqcup_r N(L_r)\) be a disjoint union of closed tubular
neighbourhoods of the components, and write \(\mu_r,\lambda_r\) for the
meridian and preferred longitude on \(\partial N(L_r)\). Because \(L\) is algebraically split, each preferred longitude is
null-homologous in the exterior of the full link.  Choose compact oriented
surfaces \(F_i,F_j,F_k\subset S^3\setminus\operatorname{int}N(L)\), in
general position, such that \(\partial F_r\) is the preferred longitude of
\(L_r\) on \(\partial N(L_r)\).  In the \(\varepsilon_r p\)-surgery solid
torus, a meridian disk has boundary
\[
\lambda_r+\varepsilon_r p\mu_r.
\]
Modulo \(p\), this boundary is \(\lambda_r\).  Thus, for each
\(r\in\{i,j,k\}\), the surface \(F_r\) caps off to a mod-\(p\)
\(2\)-cycle \(\widehat F_r\subset M\).  With the orientation convention
fixed once and for all, \(\widehat F_r\) has algebraic intersection \(+1\)
with the meridian \(m_r\) and intersection \(0\) with the other meridians.
Thus
\(
\operatorname{PD}_{\mathbb F_p}[\widehat F_r]=\alpha_r
\)
for \(r=i,j,k\). The caps lie in distinct surgery solid tori and may be chosen disjoint
from the other capped surfaces. Hence
\[
\left\langle
\alpha_i\cup\alpha_j\cup\alpha_k,[M]
\right\rangle
\]
is the mod-\(p\) triple intersection number of the closed mod-\(p\)
surfaces dual to \(\alpha_i,\alpha_j,\alpha_k\), up to the fixed
convention sign comparing cup products with oriented triple
intersections.

We then use the length-three case of Turaev's surgery
formula, in the form recalled by Cochran--Melvin: the triple
cup-product forms of a surgery manifold are calculated from the triple
Milnor invariants of the corresponding three-component sublinks of the
surgery presentation. Since \(L\) is algebraically split,
\(\bar\mu_{ijk}(L)\) is integrally defined, and reducing Turaev's formula
modulo \(p\) gives
\[
\left\langle
\alpha_i\cup\alpha_j\cup\alpha_k,[M]
\right\rangle
=
s_\mu\,\bar\mu_{ijk}(L)
\quad\text{in }\mathbb F_p.
\]
The sign \(s_\mu\) is global: it depends only on the convention for
Milnor invariants and on the convention comparing cup products with
oriented triple intersections, not on the ordered triple \(i<j<k\).
\end{proof}

\begin{theorem}[The FK tensor of an algebraically split \(p\)-surgery link]
\label{thm:FK-Milnor}
Let
\[
L=L_1\cup\cdots\cup L_m\subset S^3
\]
be an oriented algebraically split framed link with framings
\(
\varepsilon_i p\),
\( \varepsilon_i=\pm1,
\)
and let \(M=S^3_L\).  For \(i<j<k\),
\[
\kappa_p\bigl(\lambda_3^p(m_i,m_j,m_k)\bigr)
=
 s_\mu\,\varepsilon_i\varepsilon_j\varepsilon_k\,
 \bar\mu_{ijk}(L)
\quad\text{in }\mathbb F_p.
\]
More generally, let
\[
x_r=\sum_{i=1}^m a_{ri}m_i,
\qquad r=1,2,3,
\]
with \(a_{ri}\in\mathbb F_p\).  If the ordered triple
\((x_1,x_2,x_3)\) satisfies 
\(
\lambda_2(x_1,x_2)=\lambda_2(x_1,x_3)=0,
\)
then
\[
\kappa_p\bigl(\lambda_3^p(x_1,x_2,x_3)\bigr)
=
 s_\mu
 \sum_{1\le i<j<k\le m}
 \varepsilon_i\varepsilon_j\varepsilon_k\,
 \bar\mu_{ijk}(L)
 \det
 \begin{pmatrix}
 a_{1i}&a_{1j}&a_{1k}\\
 a_{2i}&a_{2j}&a_{2k}\\
 a_{3i}&a_{3j}&a_{3k}
 \end{pmatrix}.
\]
\end{theorem}

\begin{proof}
The surgery linking matrix of the framed link \(L\) is diagonal:
\[
B=\operatorname{diag}(\varepsilon_1p,\ldots,\varepsilon_mp).
\]
Thus
\(
H_1(M;\mathbb Z)\cong(\mathbb Z/p)^m,
\)
with meridian basis \(m_1,\ldots,m_m\), and the ordinary torsion linking
form is represented by \(B^{-1}\) modulo \(\mathbb Z\).  Hence
\[
\lambda_2(m_i,m_j)=0\quad (i\ne j),
\qquad
\lambda_2(m_i,m_i)=\frac{\varepsilon_i}{p}.
\]
It follows that
\(
m_i^\lambda=\varepsilon_i\alpha_i.
\) By
Theorem~\ref{thm:FK-cup},
\[
\kappa_p\bigl(\lambda_3^p(m_i,m_j,m_k)\bigr)
=
\left\langle
m_i^\lambda\cup m_j^\lambda\cup m_k^\lambda,[M]
\right\rangle=\varepsilon_i\varepsilon_j\varepsilon_k
\left\langle
\alpha_i\cup\alpha_j\cup\alpha_k,[M]
\right\rangle.
\]
Proposition~\ref{prop:Turaev-Milnor} gives
\[
\left\langle
\alpha_i\cup\alpha_j\cup\alpha_k,[M]
\right\rangle
=
 s_\mu\,\bar\mu_{ijk}(L),
\]
which proves the meridian formula.

For the determinant formula, observe that
\[
x_r^\lambda
=
\sum_{i=1}^m a_{ri}m_i^\lambda
=
\sum_{i=1}^m a_{ri}\varepsilon_i\alpha_i.
\]
By Theorem~\ref{thm:FK-cup},
\[
\kappa_p\bigl(\lambda_3^p(x_1,x_2,x_3)\bigr)
=
\left\langle
x_1^\lambda\cup x_2^\lambda\cup x_3^\lambda,[M]
\right\rangle.
\]
Expanding trilinearly and using the alternatingness of the cup product on
degree-one classes gives
\[
\sum_{1\le i<j<k\le m}
\varepsilon_i\varepsilon_j\varepsilon_k
\left\langle
\alpha_i\cup\alpha_j\cup\alpha_k,[M]
\right\rangle
\det
 \begin{pmatrix}
 a_{1i}&a_{1j}&a_{1k}\\
 a_{2i}&a_{2j}&a_{2k}\\
 a_{3i}&a_{3j}&a_{3k}
 \end{pmatrix}.
\]
Substituting the Turaev--Milnor formula completes the proof.
\end{proof}

\begin{remark}
The factor \(\varepsilon_i\varepsilon_j\varepsilon_k\) does not come from
Milnor's invariant.  It comes from torsion-linking duality:
\(
m_i^\lambda=\varepsilon_i\alpha_i.
\)
Therefore the unsigned formula
\[
\kappa_p(\lambda_3^p(m_i,m_j,m_k))=s_\mu\bar\mu_{ijk}(L)
\]
is valid only when the three framings are \(+p\).
\end{remark}

\section{Realizing triple cup tensors with fixed torsion linking form}
\label{sec:realization}

In this section we prove Theorem~\ref{thm:intro-C} by an explicit Borromean surgery construction. Fix an odd prime \(p\), and put
\[
\varepsilon_i=
\begin{cases}
+1,&1\le i\le3,\\
-1,&4\le i\le6.
\end{cases}
\]
Let
\[
\tau=\sum_{1\le i<j<k\le6}\tau_{ijk}\,
\alpha_i\wedge\alpha_j\wedge\alpha_k
\]
be the given alternating triple cup tensor.  Choose integer lifts
\(b_{ijk}\in\mathbb Z\) such that
\(
b_{ijk}\equiv \tau_{ijk}\pmod p.
\) 
We consider the oriented six-component unlink
\[
U=U_1\cup\cdots\cup U_6\subset S^3
\]
with framing \(\varepsilon_i p\), and let \(m_i\) denote the
oriented meridian of \(U_i\). For a triple \(I=(i,j,k)\), \(i<j<k\), an \(I\)-Borromean block is a
local Borromean band-sum operation on the three components
\(U_i,U_j,U_k\).  Equivalently, it is surgery on a simple one-node
\(Y\)-clasper whose three disk-leaves are meridional to
\(U_i,U_j,U_k\), in the sense of Habiro's clasper calculus
\cite{Habiro}.  We choose all supports pairwise disjoint. Recall that \(s_\mu\) denotes the fixed global sign in the Turaev--Milnor formula (Proposition~\ref{prop:Turaev-Milnor}). We orient a positive \(I=(i,j,k)\)-block so that it contributes \(+1\) to the normalized Milnor coefficient \(s_\mu\bar\mu_{ijk}\); a negative block contributes \(-1\).

For each \(i<j<k\), insert \(|b_{ijk}|\) positive \((i,j,k)\)-blocks if
\(b_{ijk}>0\), and \(|b_{ijk}|\) negative \((i,j,k)\)-blocks if
\(b_{ijk}<0\).  Let the resulting framed link be \(L_{p,\tau}\), and
set
\[
M_{p,\tau}=S^3_{L_{p,\tau}}.
\]

\begin{lemma}
\label{lem:Borromean-Milnor}
The link \(L_{p,\tau}\) is algebraically split, has framings
\(p,p,p,-p,-p,-p\), and satisfies
\[
s_\mu\bar\mu_{ijk}(L_{p,\tau})\equiv \tau_{ijk}\pmod p
\]
for every \(i<j<k\).
\end{lemma}

\begin{proof}
A Borromean block has zero pairwise linking numbers.  Inserting such a
block into a small ball meeting three components in trivial arcs changes
length-three linking data but preserves all framings and all pairwise
linking numbers.  Hence the resulting link remains algebraically split and
has the same diagonal framing matrix.

For length-three Milnor invariants, the block supported on
\((i,j,k)\) contributes \(\pm1\) to \(s_\mu\bar\mu_{ijk}\). It contributes \(0\) to
\(s_\mu\bar\mu_{abc}\) for triples \((a,b,c)\ne(i,j,k)\): Milnor's
length-three invariant of an algebraically split link depends only on the
corresponding three-component sublink, and deleting any one leaf of the local
Borromean tangle leaves a trivial tangle on the remaining components. This is also reflected in the Seifert-surface triple-intersection
interpretation \cite{MellorMelvin}, and follows directly from the
elementary effect of a simple \(Y\)-clasper in Habiro's calculus
\cite{Habiro}. Because the supports of the blocks are
disjoint, the contributions add. Therefore
\[
s_\mu\bar\mu_{ijk}(L_{p,\tau})\equiv b_{ijk}\equiv\tau_{ijk}
\pmod p.
\]
\end{proof}

\begin{lemma}
\label{lem:H1-lambda2}
For \(M_{p,\tau}\),
\(
H_1(M_{p,\tau};\mathbb Z)\cong(\mathbb Z/p)^6.
\)
With respect to the meridian basis \(m_1,\ldots,m_6\),
\[
\lambda_2(m_i,m_j)=0\quad(i\ne j),
\qquad
\lambda_2(m_i,m_i)=\frac{\varepsilon_i}{p}.
\]
Thus
\[
\lambda_2(M_{p,\tau})
\cong
\frac1p\langle1,1,1,-1,-1,-1\rangle.
\]
\end{lemma}

\begin{proof}
By Lemma~\ref{lem:Borromean-Milnor}, the surgery linking matrix is
\[
B_p=\operatorname{diag}(p,p,p,-p,-p,-p).
\]
Therefore
\(
H_1(M_{p,\tau};\mathbb Z)
\cong \operatorname{coker}(B_p)
\cong(\mathbb Z/p)^6,
\)
and the ordinary torsion linking form is represented by \(B_p^{-1}\) modulo
\(\mathbb Z\).
\end{proof}

\begin{proof}[Proof of Theorem~\ref{thm:intro-C}]
Let \(\alpha_i\) be dual to the meridian basis \(m_i\).  Lemma~\ref{lem:H1-lambda2}
gives the stated homology group and ordinary torsion linking form. By Lemma~\ref{lem:Borromean-Milnor} and Proposition~\ref{prop:Turaev-Milnor},
\[
\left\langle
\alpha_i\cup\alpha_j\cup\alpha_k,[M_{p,\tau}]
\right\rangle
=
 s_\mu\bar\mu_{ijk}(L_{p,\tau})
=
\tau_{ijk}
\]
in \(\mathbb F_p\).  This completes the proof.
\end{proof}

\begin{remark}\label{rem:coefficient-conventions}
For the constructed link \(L_{p,\tau}\), Lemma~\ref{lem:Borromean-Milnor} and
Proposition~\ref{prop:Turaev-Milnor} imply that the normalized Milnor coefficients
of the constructed link satisfy
\[
s_\mu\bar\mu_{ijk}(L_{p,\tau})\equiv \tau_{ijk}\pmod p.
\]
The corresponding meridian coefficients of
\(\kappa_p(\lambda_3^p)\) differ by the framing signs.  Indeed, since
\(
m_i^\lambda=\varepsilon_i\alpha_i,
\)
if, for \(i<j<k\), we set
\[
n_{ijk}:=\kappa_p\bigl(\lambda_3^p(m_i,m_j,m_k)\bigr),
\]
then
\[
n_{ijk}\equiv
\varepsilon_i\varepsilon_j\varepsilon_k\,\tau_{ijk}
\pmod p.
\]
\end{remark}

\section{The cup-product embedding obstruction}
\label{sec:criterion}

Let \(M\) be a rational homology \(3\)-sphere whose first homology has
exponent \(p\), and let
\[
\lambda^\#\colon H_1(M;\mathbb Z)\longrightarrow H^1(M;\mathbb F_p),
\qquad x\longmapsto x^\lambda,
\]
be the linking-duality isomorphism.  For a subgroup
\(A\subset H_1(M;\mathbb Z)\), set
\[
A^\lambda=\lambda^\#(A)
\subset H^1(M;\mathbb F_p).
\]
Because \(\lambda_2\) is nonsingular, \(\lambda^\#\) is an isomorphism. 

\begin{theorem}[Embedding obstruction in cup-product language]
\label{thm:cup-obstruction}
Let \(M\) be a rational homology \(3\)-sphere whose first homology has
exponent \(p\).  If \(M\) embeds locally flatly in \(S^4\), or in an
integer homology \(4\)-sphere, then the two Hantzsche Lagrangians
\(A,B\subset H_1(M;\mathbb Z)\) supplied by Theorem~\ref{thm:FK} satisfy
\[
\left\langle u\cup v\cup w,[M]\right\rangle=0
\]
for all \(u,v,w\in A^\lambda\), and also for all \(u,v,w\in B^\lambda\).
\end{theorem}

\begin{proof}
By Theorem~\ref{thm:FK}, the two complementary-region kernels \(A\) and
\(B\) form a dual Hantzsche pair and satisfy
\[
\lambda_3^p|_A=0,
\qquad
\lambda_3^p|_B=0.
\]
Theorem~\ref{thm:FK-cup} identifies
\(\kappa_p(\lambda_3^p)\) on such triples with the mod-\(p\) triple cup
product on the linking-dual subspace.  Because \(\kappa_p\) is an
isomorphism, vanishing of
\(\lambda_3^p\) on \(A\) and \(B\) is exactly the stated vanishing of the
triple cup tensor on \(A^\lambda\) and \(B^\lambda\).
\end{proof}

\begin{corollary}
\label{cor:no-dual-null}
Let \(M\) be a rational homology \(3\)-sphere whose first homology has
exponent \(p\).  Suppose there is no dual Hantzsche pair
\[
H_1(M;\mathbb Z)=A\oplus B
\]
for which the mod-\(p\) triple cup product vanishes on both
\(A^\lambda\) and \(B^\lambda\).  Then \(M\) admits no locally flat
embedding in \(S^4\), and no locally flat embedding in any integer
homology \(4\)-sphere.
\end{corollary}

\begin{proof}
This is the contrapositive of Theorem~\ref{thm:cup-obstruction}.
\end{proof}

For the manifolds \(M_{p,\tau}\) constructed in
Section~\ref{sec:realization}, the linking form is
\[
\lambda_2(M_{p,\tau})
\cong
\frac1p\langle1,1,1,-1,-1,-1\rangle.
\]
Equivalently, the \(\mathbb F_p\)-valued form
\(
(u,v)\longmapsto \kappa_p(\lambda_2(u,v))
\)
has matrix
\[
\langle1,1,1,-1,-1,-1\rangle.
\]
This is the split six-dimensional quadratic
space over \(\mathbb F_p\).
Under this identification, a Hantzsche Lagrangian maps to a maximal
isotropic \(3\)-plane in cohomology.  Thus the obstruction becomes purely
algebraic: the cup tensor must have a dual pair of null Lagrangians.

\begin{corollary}[Obstruction for the realized family]
\label{cor:realized-obstruction}
Let \(\tau\in\bigwedge^3(\mathbb F_p^6)^*\) be viewed as a triple cup
tensor on the split quadratic space
\((\mathbb F_p^6,\langle1,1,1,-1,-1,-1\rangle)\).  If \(\tau\) has no dual
null Lagrangian pair, then the manifold \(M_{p,\tau}\) from
Theorem~\ref{thm:intro-C} admits no locally flat embedding in \(S^4\),
and no locally flat embedding in any integer homology \(4\)-sphere.
\end{corollary}

\begin{proof}
By Theorem~\ref{thm:intro-C}, the triple cup tensor of \(M_{p,\tau}\) is
\(\tau\) in the chosen cohomology basis.  If \(M_{p,\tau}\) embedded, then
Theorem~\ref{thm:cup-obstruction} would provide a dual Hantzsche pair
whose linking-dual cohomology subspaces are both null for \(\tau\). This
contradicts the hypothesis that \(\tau\) has no dual null Lagrangian pair.
\end{proof}

\section{The spinor/Klein classification of no-dual-null tensors}
\label{sec:klein}

In this section, we prove Theorem \ref{thm:intro-E} and deduce
Corollary~\ref{cor:intro-F}.  The Klein correspondence and the two
rulings of maximal isotropic planes on the split \(6\)-dimensional
quadratic space are classical; see, for example, Hirschfeld's treatment
of finite projective geometries \cite{Hirschfeld}. Below, we recall the needed facts for classifying the specific no-dual-null condition appearing in
Corollary~\ref{cor:realized-obstruction}.

\subsection{The spinor/Klein model}

Let
\[
W=\mathbb F_p^4
\]
with ordered basis \(w_1,w_2,w_3,w_4\), and let
\[
V=\bigwedge^2 W.
\]
Fix a volume form
\[
\operatorname{vol}=w_1\wedge w_2\wedge w_3\wedge w_4
\]
and define a symmetric bilinear form on \(V\) by
\[
\langle\alpha,\beta\rangle\operatorname{vol}=\alpha\wedge\beta.
\]
Fix the basis
\[
u_1=w_1\wedge w_2,
\quad
u_2=w_1\wedge w_3,
\quad
u_3=w_1\wedge w_4,
\]
\[
u_4=w_2\wedge w_3,
\quad
u_5=w_2\wedge w_4,
\quad
u_6=w_3\wedge w_4.
\]
Observe that the only nonzero pairings among these basis vectors are
\[
\langle u_1,u_6\rangle=1,
\qquad
\langle u_2,u_5\rangle=-1,
\qquad
\langle u_3,u_4\rangle=1.
\]
Thus \(V\) is a split \(6\)-dimensional quadratic space over
\(\mathbb F_p\). For \(0\ne x\in W\), define
\[
A_{[x]}=x\wedge W.
\]
Also, for \(0\ne\varphi\in W^*\), define
\[
B_{[\varphi]}=\bigwedge^2(\ker\varphi).
\]

\begin{lemma}
\label{lem:Klein-rulings}
The maximal isotropic \(3\)-planes in \(V\) are exactly the two families
\[
A_{[x]}=x\wedge W,
\qquad [x]\in\mathbb P(W),
\]
and
\[
B_{[\varphi]}=\bigwedge^2(\ker\varphi),
\qquad [\varphi]\in\mathbb P(W^*).
\]
\end{lemma}

\begin{proof}
First, \(A_{[x]}\) has dimension \(3\), since the kernel of
\[
W\to\bigwedge^2W,
\qquad a\mapsto x\wedge a,
\]
is \(\langle x\rangle\).  It is isotropic because any two of its elements
share the factor \(x\).  Similarly, if \(H=\ker\varphi\), then
\(B_{[\varphi]}=\bigwedge^2H\) has dimension \(3\), and the wedge product
of two elements of \(\bigwedge^2H\) lies in \(\bigwedge^4H=0\).  Thus both
families consist of maximal isotropic \(3\)-planes.

Conversely, let \(L\subset\bigwedge^2W\) be a maximal isotropic
\(3\)-plane.  Since \(p\) is odd, a \(2\)-vector
\(\eta\in\bigwedge^2W\) is decomposable if and only if
\(\eta\wedge\eta=0\).  Every nonzero element of \(L\) satisfies this
equation, since \(L\) is isotropic.

Choose two linearly independent elements of \(L\).  They are decomposable
and have wedge product zero, so the corresponding \(2\)-planes in \(W\)
intersect nontrivially.  After changing basis in \(W\), we may assume
that the two elements are
\[
w_1\wedge w_2,
\qquad
w_1\wedge w_3.
\]
Let
\[
\eta=\sum_{1\le i<j\le4}a_{ij}w_i\wedge w_j\in L.
\]
The equations
\[
(w_1\wedge w_2)\wedge\eta=0,
\qquad
(w_1\wedge w_3)\wedge\eta=0
\]
force \(a_{34}=0\) and \(a_{24}=0\).  Thus
\[
\eta=
 a_{12}w_1w_2+a_{13}w_1w_3+a_{14}w_1w_4+a_{23}w_2w_3,
\]
where \(w_iw_j\) abbreviates \(w_i\wedge w_j\).  For a \(2\)-vector in dimension \(4\), decomposability is equivalent to
the Plücker relation.  In the present coordinates,
\[
\eta\wedge\eta
=
2(a_{12}a_{34}-a_{13}a_{24}+a_{14}a_{23})\operatorname{vol}.
\]
Since \(p\ne2\), \(\eta\wedge\eta=0\) is equivalent to
\[
a_{12}a_{34}-a_{13}a_{24}+a_{14}a_{23}=0.
\]
Using \(a_{34}=a_{24}=0\), this becomes
\[
a_{14}a_{23}=0.
\]
Since \(L\) is linear and \(p\ne2\), its projection to the
\((a_{14},a_{23})\)-plane is a linear subspace contained in the union of
the two coordinate axes.  Hence it is contained in one coordinate axis.
Therefore either \(a_{23}=0\) for all \(\eta\in L\), or \(a_{14}=0\) for
all \(\eta\in L\).

If \(a_{23}=0\) for all \(\eta\in L\), then
\[
L\subset w_1\wedge W=A_{[w_1]}.
\]
Both sides have dimension \(3\), so the equality holds.
If \(a_{14}=0\) for all \(\eta\in L\), then
\[
L\subset \bigwedge^2\langle w_1,w_2,w_3\rangle=B_{[w_4^*]}.
\]
Again both sides have dimension \(3\), so the equality holds.
\end{proof}

\begin{lemma}
\label{lem:Klein-complement}
For \(0\ne x\in W\) and \(0\ne\varphi\in W^*\),
\[
A_{[x]}\oplus B_{[\varphi]}=V
\quad\Longleftrightarrow\quad
\varphi(x)\ne0.
\]
\end{lemma}

\begin{proof}
If \(\varphi(x)=0\), then \(x\in\ker\varphi\).  Hence
\(
x\wedge\ker\varphi\subset A_{[x]}\cap B_{[\varphi]}
\)
is a dimension \(2\) subspace.  Therefore the two \(3\)-planes are not
complementary.

If \(\varphi(x)\ne0\), then
\[
W=\langle x\rangle\oplus\ker\varphi.
\]
Taking exterior squares on both sides gives
\[
\bigwedge^2W
=
x\wedge\ker\varphi
\oplus
\bigwedge^2(\ker\varphi).
\]
Since \(x\wedge W=x\wedge\ker\varphi\), this is precisely
\[
V=A_{[x]}\oplus B_{[\varphi]}.
\]
\end{proof}

\subsection{Three-forms and quadrics}
The fixed volume form \(\operatorname{vol}=w_1\wedge w_2\wedge w_3\wedge w_4 \) identifies \(\bigwedge^4W\) with \(\mathbb F_p\). With this identification, the spinor/Klein model gives a linear decomposition \[ \bigwedge^3(\bigwedge^2W)^* \cong \operatorname{Sym}^2(W^*)\oplus\operatorname{Sym}^2(W). \] Here \(\operatorname{Sym}^2(E)\) denotes the second symmetric power of \(E\). Since \(p\) is odd, we identify \(\operatorname{Sym}^2(W^*)\) with quadratic forms on \(W\), and \(\operatorname{Sym}^2(W)\) with quadratic forms on \(W^*\). We use the following concrete form of this isomorphism. The coordinate formulas in Appendix~\ref{app:coordinates} define quadratic forms \[ Q_\omega\in\operatorname{Sym}^2(W^*), \qquad Q_\omega^\vee\in\operatorname{Sym}^2(W). \] They satisfy \[ \omega|_{A_{[x]}}=0 \quad\Longleftrightarrow\quad Q_\omega(x)=0, \] and \[ \omega|_{B_{[\varphi]}}=0 \quad\Longleftrightarrow\quad Q_\omega^\vee(\varphi)=0. \]

They give an elementary proof of the displayed decomposition over any
field of odd characteristic.  General classifications of six-dimensional
trivectors are classical; see, for example, Revoy \cite{Revoy} and the
summary in Draisma--Shaw \cite{DraismaShaw}.  We do not need those
classifications in this paper.

For a quadratic form \(q\) over \(\mathbb F_p\), \(p\) odd, let
\[
B_q(u,v)=q(u+v)-q(u)-q(v)
\]
be its polar bilinear form.  We write
\[
\rad(q)=\{u\in E:B_q(u,v)=0\text{ for all }v\in E\},
\]
and define \(\operatorname{rank}(q)=\dim E-\dim\rad(q)\).  If
\(R=\rad(q)\), then \(q\) descends to a nondegenerate quadratic form on
\(E/R\).  A binary quadratic form is a quadratic form on a two-dimensional vector space, and it is anisotropic if it has no nonzero zero.

For a finite-dimensional vector space \(E\), write \(\mathbb P(E)\) for the projective space of one-dimensional subspaces of \(E\). If \(q\) is a quadratic form on \(E\), set \[ Z(q)=\{[v]\in\mathbb P(E):q(v)=0\}. \] Thus \(Z(q)\) is the projective zero set of \(q\). We write \(\operatorname{span}Z(q)\subset E\) for the linear span of the vectors representing the points of \(Z(q)\).

\begin{lemma}
\label{lem:quadric-span}
Let \(E\) be a \(4\)-dimensional vector space over \(\mathbb F_p\), with
\(p\) odd, and let \(q\) be a quadratic form on \(E\).  Then
\[
\dim\operatorname{span} Z(q)=2
\]
if and only if \(q\) has rank \(2\), radical a \(2\)-plane \(R\), and the
induced binary quadratic form on \(E/R\) is anisotropic.  In that case
\[
Z(q)=\mathbb P(R).
\]
For all other ranks and types, \(\operatorname{span}Z(q)\) has dimension
\(3\) or \(4\).
\end{lemma}

\begin{proof}
Let \(B_q\) be the polar bilinear form, and set
\[
R=\rad(q)=\ker B_q.
\]
Since \(p\ne2\), \(R\subset Z(q)\), and \(q\) descends to a
nondegenerate quadratic form \(\bar q\) on \(E/R\), whose dimension is
\(\operatorname{rank}(q)\). If \(\operatorname{rank}(q)=0\), then \(q=0\), so \(Z(q)=\mathbb P(E)\).
If \(\operatorname{rank}(q)=1\), then \(\dim R=3\), and the induced
one-dimensional form on \(E/R\) has no nonzero isotropic vector; hence
\(Z(q)=\mathbb P(R)\), which spans a \(3\)-dimensional subspace.

Suppose \(\operatorname{rank}(q)=2\).  Then \(\dim R=2\), and \(\bar q\)
is a nondegenerate binary quadratic form on \(E/R\).  If \(\bar q\) is
anisotropic, then the only zero in \(E/R\) is \(0\), so
\(Z(q)=\mathbb P(R)\), and the span has dimension \(2\).  If \(\bar q\) is
split, then \(Z(\bar q)\subset\mathbb P(E/R)\) consists of two distinct
points.  Their inverse images in \(E\) are two \(3\)-dimensional subspaces
containing \(R\), and together they span all of \(E\).

Suppose \(\operatorname{rank}(q)=3\).  Then \(\dim R=1\), and \(\bar q\)
is a nondegenerate ternary quadratic form.  Every nondegenerate ternary
quadratic form over \(\mathbb F_p\) is isotropic.  Indeed, after
diagonalizing it as \(aX^2+bY^2+cZ^2\), set \(Z=1\).  The two subsets
\[
a(\mathbb F_p)^2,
\qquad
-b(\mathbb F_p)^2-c
\]
of \(\mathbb F_p\) each have cardinality \((p+1)/2\), so they intersect.
Thus the ternary form has a nonzero isotropic vector.  Choose a hyperbolic
pair \(e,f\) with \(q(e)=q(f)=0\) and \(B_q(e,f)=1\) in \(E/R\) and a vector \(u\) spanning its orthogonal
complement, with \(\bar q(u)=a\ne0\).  Then
\[
e,
\qquad
f,
\qquad
-ae+f+u
\]
are isotropic and span \(E/R\).  Since \(R\subset Z(q)\), the projective
zero set spans all of \(E\).

Finally suppose \(\operatorname{rank}(q)=4\).  Since every
nondegenerate quadratic form of dimension at least \(3\) over
\(\mathbb F_p\) is isotropic, choose an isotropic vector \(e\).  By
nondegeneracy choose \(f_0\) with \(B_q(e,f_0)=1\), and replace
\(f_0\) by \(f=f_0-q(f_0)e\).  Then \(q(e)=q(f)=0\) and
\(B_q(e,f)=1\).  Let \(U=\langle e,f\rangle^\perp\). For any \(u\in U\), the
vector
\[
-q(u)e+f+u
\]
is isotropic, and
\[
u=(-q(u)e+f+u)+q(u)e-f.
\]
Thus \(U\), together with \(e\) and \(f\), lies in the span of isotropic
vectors.  Hence \(Z(q)\) spans all of \(E\).
\end{proof}

\begin{theorem}[Spinor/Klein no-dual-null criterion]
\label{thm:Klein-classification}
Let \(p\) be an odd prime, let \(W=\mathbb F_p^4\), and set
\(V=\bigwedge^2W\) with the wedge-pairing quadratic form.  Let
\(
\omega\in\bigwedge^3V^*
\)
correspond to
\[
(Q,Q^\vee)
\in
\operatorname{Sym}^2(W^*)\oplus\operatorname{Sym}^2(W).
\]
Then \(\omega\) has no dual null Lagrangian pair if and only if there is
a \(2\)-plane \(R\subset W\) such that
\[
\rad(Q)=R,
\]
the induced binary quadratic form on \(W/R\) is anisotropic,
\[
\rad(Q^\vee)=R^\perp\subset W^*,
\]
and the induced binary quadratic form on \(W^*/R^\perp\) is anisotropic.
Equivalently,
\[
Z(Q)=\mathbb P(R)\subset\mathbb P(W),
\qquad
Z(Q^\vee)=\mathbb P(R^\perp)\subset\mathbb P(W^*).
\]
\end{theorem}

\begin{proof}
Let
\[
X=Z(Q)\subset\mathbb P(W),
\qquad
Y=Z(Q^\vee)\subset\mathbb P(W^*),
\]
and set
\[
S=\operatorname{span}(X)\subset W,
\qquad
T=\operatorname{span}(Y)\subset W^*.
\]
By Lemma~\ref{lem:Klein-rulings}, every maximal isotropic \(3\)-plane is
of type \(A\) or type \(B\).  Two planes in the same ruling are never
complementary: if \(x,y\) are independent, then
\(A_{[x]}\cap A_{[y]}\) contains \(\langle x\wedge y\rangle\), and if
\(\varphi,\psi\) are independent, then
\(B_{[\varphi]}\cap B_{[\psi]}\) contains
\(\bigwedge^2(\ker\varphi\cap\ker\psi)\).  Thus a dual pair must consist
of one \(A\)-plane and one \(B\)-plane.

The condition that no null \(A\)-plane and null \(B\)-plane are
complementary is
\(
\varphi(x)=0
\)
for all \([x]\in X\) and all \([\varphi]\in Y\).  By linearity this is
equivalent to
\(
T\subseteq S^\perp.
\)

By Lemma~\ref{lem:quadric-span}, both \(S\) and \(T\) have dimension at
least \(2\).  Since
\(
\dim S^\perp=4-\dim S,
\)
the containment \(T\subseteq S^\perp\) forces
\(
\dim S=2\) and \(\dim T=2,
\)
and then \(T=S^\perp\).  Applying Lemma~\ref{lem:quadric-span} to \(Q\)
and \(Q^\vee\), we obtain the stated rank-two anisotropic description,
with \(R=S\).

Conversely, if \(Q\) and \(Q^\vee\) have the stated form, then
\[
Z(Q)=\mathbb P(R),
\qquad
Z(Q^\vee)=\mathbb P(R^\perp).
\]
Every element of \(R^\perp\) annihilates every element of \(R\), so no
null \(A\)-plane and null \(B\)-plane are complementary.
\end{proof}

\subsection{Explicit no-dual-null tensors}

Choose a nonsquare \(\delta\in\mathbb F_p^\times\).  Define
\[
\omega_\delta
=
-\delta\,u_2^*\wedge u_3^*\wedge u_6^*
+
u_2^*\wedge u_4^*\wedge u_6^*
-\delta\,u_3^*\wedge u_5^*\wedge u_6^*
+
u_4^*\wedge u_5^*\wedge u_6^*.
\]
Using the coordinate formulas of Appendix~\ref{app:coordinates}, the
associated quadrics are
\[
Q_{\omega_\delta}(x)=x_3^2-\delta x_4^2,
\qquad
Q_{\omega_\delta}^{\vee}(\varphi)=\varphi_1^2-\delta\varphi_2^2.
\]
Thus \(Q_{\omega_\delta}\) has radical
\(
R=\langle w_1,w_2\rangle
\)
and induces the anisotropic binary form \(x_3^2-\delta x_4^2\) on
\(W/R\).  Similarly, \(Q_{\omega_\delta}^\vee\) has radical
\(
R^\perp=\langle w_3^*,w_4^*\rangle
\)
and induces the anisotropic binary form
\(\varphi_1^2-\delta\varphi_2^2\) on \(W^*/R^\perp\).  Theorem
\ref{thm:Klein-classification} therefore implies that \(\omega_\delta\)
has no dual null Lagrangian pair.

\begin{corollary}
\label{cor:goodforms}
For every odd prime \(p\), there exists a split \(6\)-dimensional
quadratic space \(V\) over \(\mathbb F_p\) and a tensor
\(
\omega_p\in\bigwedge^3V^*
\)
with no dual null Lagrangian pair.
\end{corollary}

\begin{proof}
Take \(V=\bigwedge^2\mathbb F_p^4\) with the wedge-pairing form, and take
\(\omega_p=\omega_\delta\) for any nonsquare
\(\delta\in\mathbb F_p^\times\).
\end{proof}

\begin{remark}
\label{rem:diagonal-basis}
The wedge-pairing form on \(V=\bigwedge^2W\) is split. In order to apply Theorem~\ref{thm:intro-C} to the tensor \(\omega_\delta\), one needs to express the spinor/Klein model in a basis whose form is the diagonal form \[ \langle1,1,1,-1,-1,-1\rangle \] appearing in Lemma~\ref{lem:H1-lambda2}, equivalently in the meridian basis of the six-component \(\pm p\)-framed surgery construction. 
Appendix~\ref{app:diagonal} records the coefficient vector of
\(8\omega_\delta\) for such a suitably chosen dual diagonal basis. 
\end{remark}

\begin{corollary}[All odd-prime nonembedding examples]
\label{cor:all-prime}
For every odd prime \(p\), there exists a closed oriented rational
homology \(3\)-sphere \(M_p\) such that
\(
H_1(M_p;\mathbb Z)\cong(\mathbb Z/p)^6
\)
and
\[
\lambda_2(M_p)
\cong
\frac1p\langle1,1,1,-1,-1,-1\rangle
\cong H\oplus H\oplus H,
\]
where \(H\) denotes the standard hyperbolic linking form on
\((\mathbb Z/p)^2\). However, \(M_p\) admits no locally flat embedding in \(S^4\), and no
locally flat embedding in any integer homology \(4\)-sphere.
\end{corollary}

\begin{proof}
Let \(p\) be an odd prime.  By Corollary~\ref{cor:goodforms}, choose a
split \(6\)-dimensional quadratic space \(V\) over \(\mathbb F_p\) and a
form
\(
\omega_p\in\bigwedge^3V^*
\)
with no dual null Lagrangian pair.  Identify \(V\) with
\[
(\mathbb F_p^6,\langle1,1,1,-1,-1,-1\rangle).
\]
Apply Theorem~\ref{thm:intro-C}, with \(\tau=\omega_p\), to obtain
\(M_{p,\omega_p}\).  By Corollary~\ref{cor:realized-obstruction}, this
manifold admits no locally flat embedding in \(S^4\), nor in any integer
homology \(4\)-sphere.  Lemma~\ref{lem:H1-lambda2} gives the displayed
homology group and linking form.  Set \(M_p=M_{p,\omega_p}\).
\end{proof}

\appendix

\section{Coordinate formulas for the spinor/Klein construction}
\label{app:coordinates}
Throughout this appendix we use the ordered basis \(w_1,\ldots,w_4\) and the volume form \[ \operatorname{vol}=w_1\wedge w_2\wedge w_3\wedge w_4 \] fixed in Section~\ref{sec:klein}; equivalently, we identify \(\bigwedge^4W\) with \(\mathbb F_p\) by sending \(\operatorname{vol}\) to \(1\).
Let
\[
\omega=
\sum_{1\le i<j<k\le6}c_{ijk}\,
u_i^*\wedge u_j^*\wedge u_k^*
\in \bigwedge^3(\bigwedge^2W)^*.
\]

For \(x=(x_1,x_2,x_3,x_4)\in W\), set
\[
r_j=x\wedge w_j,
\qquad j=1,2,3,4.
\]
In the \(u\)-basis,
\[
r_1=(-x_2,-x_3,-x_4,0,0,0),
\]
\[
r_2=(x_1,0,0,-x_3,-x_4,0),
\]
\[
r_3=(0,x_1,0,x_2,0,-x_4),
\]
\[
r_4=(0,0,x_1,0,x_2,x_3).
\]
These vectors span \(A_{[x]}\).  Direct evaluation gives
\[
\omega|_{A_{[x]}}=0
\quad\Longleftrightarrow\quad
Q_\omega(x)=0,
\]
where
\[
Q_\omega(x)=
 c_{123}x_1^2
 +(c_{125}-c_{134})x_1x_2
 +(c_{126}-c_{234})x_1x_3
 +(c_{136}-c_{235})x_1x_4
\]
\[
+c_{145}x_2^2
+(c_{146}+c_{245})x_2x_3
+(c_{156}+c_{345})x_2x_4
+c_{246}x_3^2
\]
\[
+(c_{256}+c_{346})x_3x_4
+c_{356}x_4^2.
\]

For \(\varphi=(\varphi_1,\varphi_2,\varphi_3,\varphi_4)\in W^*\), the
following vectors span \(B_{[\varphi]}\):
\[
s_1=(0,0,0,-\varphi_4,\varphi_3,-\varphi_2),
\]
\[
s_2=(0,\varphi_4,-\varphi_3,0,0,\varphi_1),
\]
\[
s_3=(-\varphi_4,0,\varphi_2,0,-\varphi_1,0),
\]
\[
s_4=(\varphi_3,-\varphi_2,0,\varphi_1,0,0).
\]
Direct evaluation gives
\[
\omega|_{B_{[\varphi]}}=0
\quad\Longleftrightarrow\quad
Q_\omega^\vee(\varphi)=0,
\]
where
\[
Q_\omega^\vee(\varphi)
=
 c_{456}\varphi_1^2
 +(-c_{256}+c_{346})\varphi_1\varphi_2
 +(c_{156}-c_{345})\varphi_1\varphi_3
\]
\[
+(-c_{146}+c_{245})\varphi_1\varphi_4
+c_{236}\varphi_2^2
+(-c_{136}-c_{235})\varphi_2\varphi_3
\]
\[
+(c_{126}+c_{234})\varphi_2\varphi_4
+c_{135}\varphi_3^2
+(-c_{125}-c_{134})\varphi_3\varphi_4
+c_{124}\varphi_4^2.
\]

The map
\[
\omega\longmapsto(Q_\omega,Q_\omega^\vee)
\]
is a linear isomorphism
\[
\bigwedge^3(\bigwedge^2W)^*
\cong
\operatorname{Sym}^2(W^*)\oplus\operatorname{Sym}^2(W).
\]
Both sides have dimension \(20\), and the displayed formulas recover all
coefficients \(c_{ijk}\).  The mixed coefficients are recovered from
pairs of linear equations with determinant \(\pm2\), which is invertible
because \(p\) is odd.

\section{Explicit diagonal coefficient vectors}
\label{app:diagonal}

In Section~\ref{sec:klein} we constructed, for every odd prime \(p\), a no-dual-null tensor \(\omega_\delta\) in the spinor/Klein model. Here we record additional coordinate data: after changing from the spinor/Klein basis to a chosen diagonal basis compatible with the surgery construction, we give an explicit coefficient vector for the nonzero scalar multiple \(8\omega_\delta\). These coefficients can be used as the cup/Milnor coefficients in Theorem~\ref{thm:intro-C}, since multiplication by \(8\in\mathbb F_p^\times\) does not change the set of null Lagrangians. The associated meridian coefficients of \(\kappa_p(\lambda_3^p)\) are obtained from them by the framing-sign twist of Remark~\ref{rem:coefficient-conventions}.

The coefficient vector of \(8\omega_\delta\), written with respect to the
dual basis of the diagonal basis
\[
d_1=u_1+\frac12u_6,
\quad
d_2=u_2-\frac12u_5,
\quad
d_3=u_3+\frac12u_4,
\]
\[
d_4=u_1-\frac12u_6,
\quad
d_5=u_2+\frac12u_5,
\quad
d_6=u_3-\frac12u_4,
\]
is
\[
v_\delta=
[
3-6\delta,0,0,-6\delta-3,0,2\delta-1,0,0,0,-2\delta-1,
\]
\[
6\delta-3,0,0,0,-6\delta-3,0,2\delta-1,0,0,2\delta+1
],
\]
in lexicographic order
\[
123,124,125,126,134,135,136,145,146,156,
234,235,236,245,246,256,345,346,356,456.
\] Equivalently, the \(ijk\)-entry of this vector is \[ 8\omega_\delta(d_i,d_j,d_k), \qquad i<j<k. \]
This is the cup-tensor coefficient vector in the diagonal cohomology
basis.  The factor \(8\) clears the
denominators from the \(1/2\)'s in the change of basis.

For \(p\equiv3\pmod4\), we take \(\delta=-1\).  The fixed integral cup
coefficient vector is
\[
[
9,0,0,3,0,-3,0,0,0,1,
-9,0,0,0,3,0,-3,0,0,-1
].
\]
For the surgery construction with framings
\(
\varepsilon=(+,+,+,-,-,-),
\)
the corresponding meridian coefficient vector is obtained from
the cup coefficient vector by the linking-duality sign twist.  Namely,
since
\[
m_i^\lambda=\varepsilon_i\alpha_i,
\]
a cup coefficient \(\tau_{ijk}\) gives meridian coefficient
\[
n_{ijk}=\varepsilon_i\varepsilon_j\varepsilon_k\tau_{ijk}.
\]
For \(\delta=-1\), this gives the meridian coefficient vector
\[
[
9,0,0,-3,0,3,0,0,0,1,
9,0,0,0,3,0,-3,0,0,1
].
\]
Thus the fixed integral cup coefficient vector displayed above is the
cup/Milnor coefficient vector used for the Borromean construction, while
the final vector is the corresponding meridian coefficient vector for
\(\kappa_p(\lambda_3^p)\), obtained
by the framing-sign twist of Remark~\ref{rem:coefficient-conventions}.

\bigskip
\noindent\textsc{Yau Mathematical Sciences Center, Tsinghua University}\\
\textit{Email:} \texttt{weizheniu@mail.tsinghua.edu.cn}

\end{document}